\documentclass [12pt, reqno]{amsart}

\usepackage{times}

\usepackage{amssymb,latexsym,amsmath,amsfonts, eucal}
\usepackage[usenames,dvipsnames,svgnames,table]{xcolor}
\usepackage{mathrsfs}

\newtheorem{thm}{Theorem}[section]
      \newtheorem{lemma}[thm]{Lemma}

      \newtheorem{defn}[thm]{Definition}
      \newtheorem{examp}[thm]{Example}

      \numberwithin{equation}{section}

\title [de Branges matrices and related topics]{de Branges matrices and associated de Branges spaces of vector valued entire functions}
\author[ Mahapatra]{Subhankar Mahapatra}

\address{
	Department of Mathematics\\
	Indian Institute of Technology Ropar\\
	140001\\
	India}

\email{
\begin{tabular}[t]{l}
 subhankar.19maz0001@iitrpr.ac.in\\
subhankarmahapatra95@gmail.com \\
\end{tabular}
}

\author[Sarkar]{Santanu Sarkar}

\address{
	Department of Mathematics\\
	Indian Institute of Technology Ropar\\
	140001\\
	India}

\email{
\begin{tabular}[t]{l}
 santanu@iitrpr.ac.in\\
santanu87@gmail.com \\
\end{tabular}
}

\begin{document}
	


\begin{abstract}
\noindent This paper extends the concept of de Branges matrices to any finite $m\times m$ order where $m=2n$. We shall discuss these matrices along with the theory of de Branges spaces of $\mathbb{C}^n$-valued entire functions and their associated functions. A parametrization of these matrices is obtained using the Smirnov maximum principle for matrix valued functions. Additionally, a factorization of matrix valued meromorphic functions is discussed.
\end{abstract}

\maketitle
\section{Introduction}
\label{Sec-1}
This paper extends the concept of de Branges matrices introduced by Golinskii and Mikhailova in \cite{Golinskii} to any finite $m\times m$ order where $m=2n$. Their primary motivation was to investigate the close connection between Hilbert spaces of entire functions and the $J$ theory of analytic matrix valued functions. They examined various aspects of de Branges spaces of scalar valued entire functions from this perspective. A detailed study of de Branges spaces of scalar valued entire functions is available in \cite{Branges 4}, while de Branges spaces of vector valued entire functions and $J$-contractive matrix valued functions are discussed in \cite{ArD08} and \cite{ArD18}. Our motivation is to explore de Branges spaces of vector valued entire functions independently from the viewpoint of de Branges matrices. We believe that the present work will be applicable to interpolation theory (see \cite{ArD01}) and functional model theory (see \cite{Gubreev}).\\
In section \ref{Sec-2}, we review definitions of various classes of matrix valued holomorphic functions, such as the Carathéordory class and the Smirnov class. We highlight special properties of these classes, including the integral representation of elements in the Carathéordory class and the Smirnov maximum principle for elements in the Smirnov class. Additionally, we discuss the Potapov-Ginzburg transform for $J$-contractive matrix valued functions. In section \ref{Sec-3}, we revisit de Branges spaces of vector valued entire functions and their associated functions. Section \ref{Sec-4} addresses the problem of identifying a common factor for multiple matrix valued meromorphic functions that encompasses all of their poles. Section \ref{Sec-5} presents the proposed extension of de Branges matrices. Also, an analog representation of de Branges matrices, originally known as the real representation of de Branges matrices, has been established. Some examples of these de Branges matrices are also included. Finally, in the last section, we parametrize these de Branges matrices.\\
Along with the standard notations, the following nonstandard notations will be used throughout this paper:\\
$f^{\#}(z)=f(\overline{z})^*$ and $f^{-*}(z)=\{f(z)^*\}^{-1}$,\\
$(H^2_{n\times n})^\perp=\{f:f^{\#}\in H^2_{n\times n}\}$ and $(H^2_n)^\perp=(H^2_{n\times 1})^\perp$,\\
$\rho_w(z)=-2\pi i(z-\overline{w})$.

\section{Preliminaries}
\label{Sec-2}
In this section, we revisit some results concerning matrix valued holomorphic functions that will be utilized in the later sections. Although these results can be found in \cite{ArD08}, we include them here for the convenience of our readers. First, we recall some well-known classes of matrix valued holomorphic functions.\\
\textbf{Hardy Hilbert spaces:} $H^2_{n\times n}$ denotes the class of $n\times n$ matrix valued functions $f(z)$ holomorphic in $\mathbb{C}_+$ and satisfying
$$||f||_2^2=\mbox{sup}_{y>0}\int_{-\infty}^\infty \mbox{trace}\{f(x+iy)^*f(x+iy)\}dx<\infty.$$
$H^2_n$ denotes the same class for $\mathbb{C}^n$-valued holomorphic functions in $\mathbb{C}_+$, i.e., $H^2_n=H^2_{n\times 1}$. It is known that these classes are Hilbert spaces.\\
\textbf{Carathéordory class:} $\mathcal{C}^{n\times n}$ denotes the class of $n\times n$ matrix valued functions $f(z)$ holomorphic in $\mathbb{C}_+$ such that the real part of $f$ is positive semi-definite for all $z\in\mathbb{C}_+$, i.e.,
$$\mbox{Re}~f(z)=\frac{f(z)+f(z)^*}{2}\succeq 0\hspace{.3cm}\mbox{for all}~z\in\mathbb{C}_+.$$
A function $f(z)\in \mathcal{C}^{n\times n}$ if and only if it can be represented as the following integral form
\begin{equation}
f(z)=iQ-izP+\frac{1}{\pi i}\int_{-\infty}^\infty \left\{\frac{1}{x-z}-\frac{x}{1+x^2}\right\}~d\sigma(x),
\end{equation}
where $Q=Q^*$, $P\succeq 0$ are $n\times n$ complex matrices and $\sigma(x)$ is a nondecreasing $n\times n$ matrix valued function on $\mathbb{R}$ such that $\int_{-\infty}^\infty \frac{d(\mbox{trace}~\sigma(x))}{1+x^2}<\infty$.\\
\textbf{Smirnov class:} $\mathcal{N}_+^{n\times n}$ denotes the class of $n\times n$ matrix valued functions $f(z)$ holomorphic in $\mathbb{C}_+$ such that it can be represented as
$$f(z)=h(z)^{-1}g(z),$$
where $g(z)$ is an $n\times n$ matrix valued bounded holomorphic function in $\mathbb{C}_+$ and $h(z)$ is a scalar valued bounded outer function in $\mathbb{C}_+$. This class is closed under addition and suitable matrix multiplication. The Smirnov maximum principle (see \cite{ArD08}, Theorem $3.59$), which will be utilized in the final section, is one of the important properties of this class.\\
Additionally, $\mathcal{S}_{in}^{n\times n}$ denotes the class of $n\times n$ matrix valued functions $f(z)$ holomorphic in $\mathbb{C}_+$ such that 
$$||f(z)||\leq 1\hspace{.3cm}\mbox{for all}~z\in\mathbb{C}_+,$$
and the corresponding boundary function is unitary almost everywhere on $\mathbb{R}$, i.e.,
$$f(x)f(x)^*=I_n\hspace{.3cm}\mbox{for a.e.}~x\in\mathbb{R}.$$
Suppose $J$ is a signature matrix of order $m$, i.e., $J=J^*=J^{-1}$. Consider the orthogonal projection matrices
$$P=\frac{I_m+J}{2}\hspace{.3cm}\mbox{and}\hspace{.3cm}Q=\frac{I_m-J}{2}.$$
We denote $P(J)$, the family of $m\times m$ matrix valued meromorphic functions $\mathcal{A}(z)$ in $\mathbb{C}_+$ such that $\mathcal{A}(z)^*J\mathcal{A}(z)\preceq J$, where $\mathcal{A}(z)$ is holomorphic. Now, if $\mathcal{A}\in P(J)$ and $\mathcal{A}(x)^*J\mathcal{A}(x)= J$ for almost every $x\in\mathbb{R}$, we call $\mathcal{A}$ to be $J$-inner. We denote the family of $m\times m$ $J$-inner matrix valued functions as $U(J)$. If $\mathcal{A}\in P(J)$ the corresponding Potapov-Ginzburg transform is given by
$$\mbox{PG}(\mathcal{A})(z)=[P\mathcal{A}(z)+Q]~[P+Q\mathcal{A}(z)]^{-1}=[P-\mathcal{A}(z)]^{-1}[\mathcal{A}(z)P-Q].$$
Note that $\mathcal{A}\in U(J)$ if and only if $\mbox{PG}(\mathcal{A})\in \mathcal{S}^{m\times m}_{in}$.

\section{de Branges spaces and associated functions}
\label{Sec-3}
This section briefly recalls de Branges spaces $\mathcal{B}(\mathfrak{E})$ of vector valued entire functions corresponding to an $n\times 2n$ matrix valued entire function 
\begin{equation}
\mathfrak{E}(z)=[E_-(z)~E_+(z)],\label{de Branges pair}
\end{equation}
where $E_-$ and $E_+$ are blocks of order $n\times n$. An $n\times 2n$ matrix valued entire function of the form $(\ref{de Branges pair})$ generates a de Branges space $\mathcal{B}(\mathfrak{E})$ consisting of $\mathbb{C}^n$-valued entire functions and having the reproducing kernel
\begin{equation}
K_w(z):= \left\{
    \begin{array}{ll}
         \frac{E_+(z)E_+(w)^*-E_-(z)E_-(w)^*}{\rho_w(z)}  & \mbox{if } z \neq \overline{w} \\
         \frac{E_+^{'} (\overline{w})E_+(w)^*- E_-^{'}(\overline{w})E_-(w)^*}{-2\pi i} & \mbox{if } z = \overline{w}
    \end{array} \right.\label{de Branges kernel}
\end{equation}
if the blocks $E_-(z)$ and $E_+(z)$ are satisfying the following conditions
\begin{enumerate}
\item $\det E_+(z)\not\equiv 0$ in $\mathbb{C}$, and
\item $E_+^{-1}E_-\in\mathcal{S}^{n\times n}_{in}$.
\end{enumerate}
Since $E_+^{-1}E_-\in\mathcal{S}^{n\times n}_{in}$, the following relations between $E_+$ and $E_-$ are immediate
\begin{equation}
E_+(z)E_+(z)^*-E_-(z)E_-(z)^*\succeq 0\hspace{.3cm}\mbox{for}~z\in\mathbb{C}_+,
\end{equation}
and 
\begin{equation}
E_+(x)E_+(x)^*-E_-(x)E_-(x)^*=0\hspace{.3cm}\mbox{for}~x\in\mathbb{R}.\label{Equality on real line}
\end{equation}
The following theorem (see \cite{ArD18}, Theorem $3.10$) describes the elements in the space $\mathcal{B}(\mathfrak{E})$ and the associated inner product.
\begin{thm}
Let $\mathcal{B}(\mathfrak{E})$ be a de Branges space corresponding to the $n\times 2n$ matrix valued entire function of the form $(\ref{de Branges pair})$, then
\begin{equation}
\mathcal{B}(\mathfrak{E})=\{f:\mathbb{C}\to\mathbb{C}^n~\mbox{entire}~:E_+^{-1}f\in H_n^2~\mbox{and}~E_-^{-1}f\in (H_n^2)^\perp\}
\end{equation}
and the inner product for any $f,g\in \mathcal{B}(\mathfrak{E})$ is given by
\begin{equation}
\langle f,g\rangle_{\mathcal{B}(\mathfrak{E})}=\int_{-\infty}^\infty g(x)^*\{E_+(x)E_+(x)^*\}^{-1}f(x)~dx.
\end{equation}
\end{thm}
A characterization of the de Branges space $\mathcal{B}(\mathfrak{E})$ is presented in \cite{JFA} (Theorem $7.1$). Recently, we have extended this characterization to de Branges spaces based on operator valued reproducing kernels in \cite{Mahapatra}.\\
Recall that an $n\times n$ matrix valued entire function $S(z)$ is said to be associated with the de Branges space $\mathcal{B}(\mathfrak{E})$ if $\det S(w)\neq 0$ for some $w\in\mathbb{C}$ and for any $f\in \mathcal{B}(\mathfrak{E})$,
$$\frac{f(z)-S(z)S(w)^{-1}f(w)}{z-w}\in \mathcal{B}(\mathfrak{E}).$$
For a discussion on the associated functions of de Branges spaces of scalar valued entire functions, refer to \cite{Branges 4}; for vector valued entire functions, see \cite{Branges 5}.
The following theorem characterizes associated functions of $\mathcal{B}(\mathfrak{E})$ and describes a bounded operator on  $\mathcal{B}(\mathfrak{E})$.
\begin{thm}
\label{Associated function}
Let $\mathcal{B}(\mathfrak{E})$ be a de Branges space corresponding to the $n\times 2n$ matrix valued entire function of the form $(\ref{de Branges pair})$, and $S(z)$ is an $n\times n$ matrix valued entire function such that 
\begin{equation}
\frac{E_+^{-1}S}{\rho_i}\in H^2_{n\times n}\hspace{.3cm}\mbox{and}\hspace{.3cm}\frac{E_-^{-1}S}{\rho_{-i}}\in(H^2_{n\times n})^\perp.\label{Associated function condition}
\end{equation}
Then $E_+^{-1}S$ and $E_-^{-1}S$  are holomorphic in $\overline{\mathbb{C}_+}$ and $\overline{\mathbb{C}_-}$ respectively. Moreover, if $\det S(w)\neq 0$ for some $w\in\mathbb{C}$, the linear transformation $R_S(w):\mathcal{B}(\mathfrak{E})\to\mathcal{B}(\mathfrak{E})$ defined by
\begin{equation}
(R_S(w)f)(z)=\frac{f(z)-S(z)S(w)^{-1}f(w)}{z-w}\hspace{.3cm}\mbox{for}~f\in \mathcal{B}(\mathfrak{E})
\end{equation}
is everywhere defined bounded linear operator on $\mathcal{B}(\mathfrak{E})$.

Conversely, suppose $\det K_\beta(\beta)\neq 0$ for some $\beta\in\mathbb{C}$ and $\mathcal{B}(\mathfrak{E})$ is invariant under $R_S(w)$. Then $(\ref{Associated function condition})$ holds.
\end{thm}
\begin{proof}
The proof of this theorem follows from Lemma $3.18$ in \cite{ArD18}.
\end{proof}
\section{Factorization of matrix valued meromorphic functions}
\label{Sec-4}
 In this section, we discuss a factorization of matrix valued meromorphic functions in $\mathbb{C}$. Specifically, we decompose multiple matrix valued meromorphic functions to identify one common factor that encompasses all the poles of the original functions.  Suppose $F(z)$ is an $n\times n$ matrix valued meromorphic function in $\mathbb{C}$ and $\det F(z)\not\equiv 0$. A point $z_0\in\mathbb{C}$ is a pole of $F(z)$ if it is a pole of one of its entries, and $z_0$ is a zero of $F(z)$ if it is a pole of $F(z)^{-1}$. For any $z_0\in\mathbb{C}$, $F(z)$ can be decomposed into the following form (see, \cite{Gant}, sections $VI.2$, $VI.3$)
 $$F(z)=M(z)\mbox{diag}((z-z_0)^{r_1}\ldots(z-z_0)^{r_n})N(z),$$
where $M(z)$ and $N(z)$ are analytic and invertible at $z_0$ and $\{r_1, r_2,\dots, r_n\}$ is an ascending sequence of integers. When $r_j<0$, the numbers $|r_j|$ are called the partial pole multiplicities of $F(z)$ at $z_0$, and when $r_j>0$, the numbers $r_j$ are called the partial zero multiplicities of $F(z)$ at $z_0$. Now, we recall the ideas of eigenvector and pole vector of $F(z)$ at any point $z_0\in\mathbb{C}$. A more detailed discussion of them and the other related results can be found in \cite{Bart1} (Chapter $2$, section $1$). Also, a factorization of meromorphic matrix-valued functions of finite order can be found in \cite{Ran}.\\
 A nonzero vector $u_1\in\mathbb{C}^n$ is called an eigenvector of $F(z)$ at the zero $z_0$ if there exist vectors $\{u_2, u_3,\ldots\}\subset\mathbb{C}^n$ such that $F(z)\sum_{j=0}^\infty u_{j+1}(z-z_0)^j$ is analytic at $z_0$ and has a zero at $z_0$. If this zero has order at least $m$, then $u_1, u_2,\ldots, u_m$ is called a zero chain of length $m$ of $F(z)$ at $z_0$. It can be proved that the number of independent eigenvectors at $z_0$ equals the number of partial zero multiplicities. Furthermore, for a given eigenvector $u_1$, the maximal length of a zero chain starting at $u_1$ corresponds to one of the partial zero multiplicities.\\
 A nonzero vector $v_1\in\mathbb{C}^n$ is called a pole vector of $F(z)$ at the pole $z_0$ if there exist vectors $\{v_2, v_3,\ldots\}\subset\mathbb{C}^n$ such that $F(z)^{-1}\sum_{j=0}^\infty v_{j+1}(z-z_0)^j$ is analytic at $z_0$ and has a zero at $z_0$. If this zero has order at least $l$, then $v_1, v_2, \ldots, v_l$ is called a pole chain of length $l$ of $F(z)$ at $z_0$. It can be proved that the number of independent pole vectors at $z_0$ equals the number of partial pole multiplicities. Furthermore, for a given pole vector $v_1$, the maximal length of a pole chain starting at $v _1$ corresponds to one of the partial pole multiplicities.\\
 The following theorem is a matrix analog of Theorem $19$ in the appendix of \cite{Rovnyak}, and the proof can be done similarly. 
 \begin{thm}
 Let $\{P_k\}_{k=1}^\infty$ be a sequence of orthogonal projection matrices of order $n\times n$ and $\{z_k\}_{k=1}^\infty$ be a sequence of nonzero complex numbers such that $|z_k|\to\infty$ as $k\to\infty$. Then 
 \begin{multline}
 \label{Product}
 P(z)=\lim_{k\to\infty}\exp(\frac{z}{z_k}P_k+\ldots+\frac{1}{k}\frac{z^k}{z_k^k}P_k)(I_n-\frac{z}{z_k}P_k)\ldots\\
 \exp(\frac{z}{z_1}P_1)(I_n-\frac{z}{z_1}P_1)
 \end{multline}
 converges uniformly in any bounded set with respect to the operator norm and $P(z)$ is an $n\times n$ matrix valued entire function. Also, $\det P(z)\neq 0$ for all $z\in\mathbb{C}\setminus\{z_k\}_{k=1}^\infty$.
 \end{thm}
The following theorem gives a factorization of a meromorphic matrix valued function. 
\begin{thm}
Let $F(z)$ be an $n\times n$ matrix valued meromorphic function such that $\det F(0)\neq 0$. Then $$G(z)=P(z)~F(z),$$ where $P(z)$ is an $n\times n$ matrix valued entire function of the form $(\ref{Product})$, and $G(z)$ is an $n\times n$ matrix valued entire function.
\end{thm}
\begin{proof}
If $F(z)$ has no pole, the theorem follows with $P(z)=I_n$ and $G(z)=F(z)$ for all $z\in\mathbb{C}$. Otherwise, let $z_1\neq 0$ be a pole of $F(z)$ nearest to the origin. We denote $N_1$ as the linear span of the pole vectors of $F(z)$ at $z=z_1$, and $P_1$ is the orthogonal projection on $N_1$. Suppose $\dim P_1=r\leq n$ and 
$P_1=
\begin{bmatrix}
    I_r  & 0 \\
    0  & 0 \\
\end{bmatrix}
$
with respect to some orthonormal basis. Now, define
\begin{equation}
\tilde{G}_1(z)=(I_n-\frac{z}{z_1}P_1)~F(z).\label{step 1}
\end{equation}
We denote the partial pole multiplicities of $F(z)$ at $z=z_1$ by $|k_1|,\ldots,|K_l|~(l\leq n)$. We now claim that the partial pole multiplicities of $\tilde{G}_1(z)$ at $z=z_1$ are given by the nonzero numbers among $|k_1+1|, \ldots, |k_l+1|$. Now, suppose that $\tilde{v}_1,\ldots, \tilde{v}_m$ is a pole chain of $\tilde{G}_1(z)$ at $z_1$, that is $\tilde{v}_1\neq 0$ and there exist vectors $\tilde{v}_{m+1},\ldots$ such that
\begin{equation}
\tilde{G}_1(z)^{-1}\sum_{j=1}^\infty\tilde{v}_j(z-z_1)^{j-1}=\sum_{j=m}^\infty\psi_j(z-z_1)^j.\label{step 2}
\end{equation}
Now, from equation $(\ref{step 1})$, we have
$$\tilde{G}_1(z)^{-1}=F(z)^{-1}(I_n-\frac{z}{z_1}P_1)^{-1}=F(z)^{-1}
\begin{bmatrix}
    z_1(z_1-z)^{-1}I_r  & 0 \\
    0  & I_{n-r} \\
\end{bmatrix}.
$$
Now, we write the vectors $\tilde{v}_j=\tilde{v}_{j1}+\tilde{v}_{j2}$ as the orthogonal sum, where $\tilde{v}_{j1}\in N_1$ and $\tilde{v}_{j2}\in N_1^\perp$. Then from $(\ref{step 2})$, we can write
\begin{align*}
\sum_{j=m}^\infty\psi_j(z-z_1)^j&=F(z)^{-1}\{\sum_{j=1}^\infty -z_1\tilde{v}_{j1}(z-z_1)^{j-2}+\sum_{j=1}^\infty\tilde{v}_{j2}(z-z_1)^{j-1}\}\\
&=F(z)^{-1}(z-z_1)^{-1}\{-z_1\tilde{v}_{11}+\sum_{j=1}^\infty(-z_1\tilde{v}_{(j+1) 1}+\tilde{v}_{j2})(z-z_1)^j\}.
\end{align*}
Here $\tilde{v}_{11}\neq 0$ since otherwise $-z_1\tilde{v}_{21}+\tilde{v}_{12}$ would be a pole vector of $F(z)$ at $z_1$, which can be true only if $\tilde{v}_{12}=0$. But if both $\tilde{v}_{11}=0$ and $\tilde{v}_{12}=0$ then $\tilde{v}_1=0$, contradicting the assumption $\tilde{v}_1\neq 0$. Thus it follows that the vectors 
$$v_1=z_1\tilde{v}_{11}, v_2=-z_1\tilde{v}_{21}+\tilde{v}_{12}, \ldots, v_{m+1}=-z_1\tilde{v}_{(m+1) 1}+\tilde{v}_{m2}$$
form a pole chain of length $m+1$ of $F(z)$ at $z_1$.\\
Conversely, let $v_1, v_2, \ldots, v_{m+1}~(m\geq 1)$ is a pole chain of $F(z)$ at $z_1$. Then $v_1\neq 0$ and 
$$F(z)^{-1}\sum_{j=1}^\infty v_j(z-z_1)^{j-1}=\sum_{j=m+1}^\infty\phi_j(z-z_1)^j.$$
Since $F(z)^{-1}=\tilde{G}_1(z)(I_n-\frac{z}{z_1}P_1)$, we have
\begin{align*}
\sum_{j=m+1}^\infty\phi_j(z-z_1)^j&=\tilde{G}_1(z)^{-1}\{\sum_{j=1}^\infty-\frac{1}{z_1}v_{j1}(z-z_1)^j+\sum_{j=1}^\infty v_{j2}(z-z_1)^{j-1}\}\\
&=\tilde{G}_1(z)^{-1}(z-z_1)\{v_{12}+\sum_{j=1}^\infty(-\frac{1}{z_1}v_{j1}+v_{(j+1) 2})(z-z_1)^{j-1}\}.
\end{align*}
Since $v_1$ belongs to $N_1$ , $v_{12}=0$. Thus we have a pole chain 
$$\tilde{v}_1=-\frac{1}{z_1}v_{11}+v_{22},\ldots, \tilde{v}_m=-\frac{1}{z_1}v_{m1}+v_{(m+1) 2}$$
of length $m$ of $\tilde{G}_1(z)$ at $z_1$. Due to the correspondence between the partial pole multiplicities and the lengths of the pole chains, the claim follows. Next we define 
$$G_1(z)=\exp(\frac{z}{z_1}P_1)(I_n-\frac{z}{z_1}P_1)F(z).$$
Clearly, $G_1(z)$ has the same partial pole multiplicities as $\tilde{G}_1(z)$. Now, if $G_1(z)$ is entire, consider $P(z)=\exp(\frac{z}{z_1}P_1)(I_n-\frac{z}{z_1}P_1)$ and $G(z)=G_1(z)$. Otherwise, let $z_2$ be a number nearest to the origin such that $G_1(z)$ has a pole at $z_2$ and continue inductively. If the number of poles of $F(z)$ is finite, this process will stop after finite steps, and we will get the desired factorization. Now, suppose the number of poles of $F(z)$ is infinite, i.e., $\{z_k\}_{k=1}^\infty$ such that $|z_k|\to\infty$ as $k\to\infty$. Then, using the previous theorem, we conclude that a matrix valued entire function $P(z)$ of the form $(\ref{Product})$ exists. Also, it can be shown that $G(z)=\lim_{k\to\infty}G_k(z)$ converges uniformly on every bounded set with respect to the operator norm. This completes the proof.
\end{proof}
The following theorem is an extended version of the previous theorem. Here, we simultaneously factorize two matrix valued meromorphic functions in $\mathbb{C}$.
\begin{thm}
\label{Factorization}
Let $A(z)$ and $B(z)$ be two $n\times n$ matrix valued meromorphic functions such that $\det A(0)\neq 0$ and $\det B(0)\neq 0$. Then there exists an $n\times n$ matrix valued entire function $P(z)$ of the form $(\ref{Product})$ such that
$$\tilde{A}(z)=P(z)A(z);\hspace{.3cm}\tilde{B}(z)=P(z)B(z),$$
where $\tilde{A}(z)$ and $\tilde{B}(z)$ are $n\times n$ matrix valued entire functions.
\end{thm}
\begin{proof}
Suppose $z_1$ is a nonzero complex number nearest to the origin at which at least one of $A(z)$ or $B(z)$ has a pole. We denote $N_1(A;z_1)$ and $N_1(B;z_1)$ as the linear span of the pole vectors of $A(z)$ and $B(z)$, respectively, at $z_1$. According to our consideration, at least one of these two sets is nonempty. Now, we consider $N_1$ to be the linear span of the union of $N_1(A;z_1)$ and $N_1(B;z_1)$, and $P_1$ is the orthogonal projection matrix on $N_1$. Suppose $\dim P_1=r\leq n$ and $P_1=
\begin{bmatrix}
    I_r  & 0 \\
    0  & 0 \\
\end{bmatrix}
$ 
with respect to some orthonormal basis. For the definiteness, let us assume that $N_1(A;z_1)$ is nonempty, i.e., $A(z)$ has pole at $z_1$. Now, define
\begin{equation}
\tilde{A}_1(z)=(I_n-\frac{z}{z_1}P_1)~A(z).
\end{equation}
We denote the partial pole multiplicities of $A(z)$ at $z=z_1$ by $|k_1|,\ldots,|K_l|~(l\leq n)$. We now claim that the partial pole multiplicities of $\tilde{A}_1(z)$ at $z=z_1$ are given by the nonzero numbers among $|k_1+1|, \ldots, |k_l+1|$. Since $N_1(A;z_1)\subseteq N_1$, the previous claim, along with its converse, can be proved as in the previous theorem. Now, we define
$$A_1(z)=\exp(\frac{z}{z_1}P_1)(I_n-\frac{z}{z_1}P_1)A(z).$$
Similarly, we define 
$$B_1(z)=\exp(\frac{z}{z_1}P_1)(I_n-\frac{z}{z_1}P_1)B(z).$$
Note that if $N_1(B;z_1)$ is nonempty, the above definition of $B_1(z)$ can be justified as in the case of $A_1(z)$ and the change of partial pole multiplicities between $B_1(z)$ and $B(z)$ can be observed. If $N_1(B;z_1)$ is empty, in that case, the definition of $B_1(z)$ can still be justified because there is no difference of poles and partial pole multiplicities between $B_1(z)$ and $B(z)$. Now, if we continue this process inductively, the desired factorizations can be obtained as in the previous theorem.
\end{proof}

\section{de Branges matrices and examples}
\label{Sec-5}
This section extends the idea of de Branges matrices introduced by Golinskii and Mikhailova in \cite{Golinskii}, to study them in connection with the de Branges spaces of vector valued entire functions. Some examples of de Branges matrices are also discussed here. We consider two $m\times m$ $(m=2n)$ signature matrices $j_m$ and $\mathscr{J}_m$ satisfying the condition $M^*\mathscr{J}_m M=j_m$, where
\begin{equation}
j_m=\begin{bmatrix}
    I_n & 0 \\
    0  & -I_n \\
\end{bmatrix};\hspace{.2cm}
\mathscr{J}_m=\begin{bmatrix}
    0 & iI_n \\
    -iI_n  & 0 \\
\end{bmatrix};\hspace{.2cm}
M=\frac{1}{\sqrt{2}}\begin{bmatrix}
    iI_n & -iI_n \\
    I_n  & I_n \\
\end{bmatrix}.
\end{equation}
Let 
\begin{equation}
\label{Basic matrix}
\mathcal{A}(z)=\begin{bmatrix}
    a_{11}(z) & a_{12}(z) \\
    a_{21}(z)  & a_{22}(z) \\
\end{bmatrix}
\end{equation}
be an $m\times m$ matrix valued meromorphic function in $\mathbb{C}_+$ and 
\begin{equation}
\label{Secondary matrix}
U(z)=\mathcal{A}(z)M=\begin{bmatrix}
    u_{11}(z) & u_{12}(z) \\
    u_{21}(z)  & u_{22}(z) \\
\end{bmatrix},
\end{equation}
where the entries $a_{rt}(z)$ and $u_{rt}(z)$ are of size $n\times n$.
The following lemma describes the intimate connections between the entries of the $\mathscr{J}_m$-contractive matrix valued meromorphic functions in $\mathbb{C}_+$.
\begin{lemma}
\label{Meromorphic case}
Let $\mathcal{A}(z)$ be an element of the class $P(\mathscr{J}_m)$. Then the following implications hold:
\begin{enumerate}
\item $U(z)^*\mathscr{J}_m U(z)\leq j_m$ for all $z\in\mathbb{C}_+$, where $\mathcal{A}(z)$ is holomorphic.
\item $u_{12}(z)$ is invertible for all $z\in\mathbb{C}_+$, where $\mathcal{A}(z)$ is holomorphic.
\item The $n\times n$ matrix valued functions $\Phi(z)=-iu_{22}(z)u_{12}^{-1}(z)\in\mathcal{C}^{n\times n}$ and $\chi(z)=u_{12}^{-1}(z)u_{11}(z)\in\mathcal{S}^{n\times n}$.
\item The $n\times n$ matrix valued function $\frac{u_{12}^{-1}}{\rho_i}\in H_{n\times n}^2$.
\end{enumerate}
Moreover, if $\mathcal{A}(z)$ belongs to $U(\mathscr{J}_m)$, then
\begin{enumerate}
\item[(5)] The $n\times n$ matrix valued function $\chi(z)=u_{12}^{\#}(z)\{u_{11}^{\#}(z)\}^{-1}\in\mathcal{S}^{n\times n}_{in}$.
\item[(6)] The $n\times n$ matrix valued function $\frac{-i\{u_{11}^{\#}\}^{-1}}{\rho_i}\in H_{n\times n}^2$.
\end{enumerate}
\end{lemma}
\begin{proof}
The proof of this lemma is similar to Lemma $4.35$ in \cite{ArD08}.
\end{proof}
Now, we extend the definition of de Branges matrices for $m\times m$ matrix valued functions, which was hinted in \cite{ArD01} (section $7$).
\begin{defn}
Let $\mathcal{A}(z)$ belongs to the class $U(\mathscr{J}_m)$. Then $\mathcal{A}(z)$ is  said to be a de Branges matrix if the $n\times n$ matrix valued function
\begin{equation}
\Phi(z)=-iu_{22}(z)u_{12}^{-1}(z)=[a_{22}(z)-ia_{21}(z)]~[a_{11}(z)+ia_{12}(z)]^{-1}\label{Main function}
\end{equation}
is holomorphic in $\mathbb{R}$.
\end{defn}
 \begin{examp}
 \label{Example 1 }
 Let $\mathcal{A}(z)$ belongs to the class $U(\mathscr{J}_m)$ and holomorphic in $\mathbb{R}$. Then $\mathcal{A}(z)$ is a de Branges matrix. Indeed, let $\mu\in\mathbb{R}$ be a pole of $\Phi(z)$. Then $a_{11}(\mu)+ia_{12}(\mu)$ is not invertible, as $\mathcal{A}(z)$ is holomorphic at $\mu$. Thus $\det(a_{11}(\mu)+ia_{12}(\mu))=0$ and $\det(a_{11}^*(\mu)-ia_{12}^*(\mu))=0$. Also, since $\mathcal{A}(z)$ is $\mathscr{J}_m$-inner, we have 
 $$a_{11}(\mu)a_{12}^*(\mu)=a_{12}(\mu)a_{11}^*(\mu),$$
 which implies that $\det(a_{11}(\mu)a_{11}^*(\mu)+a_{12}(\mu)a_{12}^*(\mu))=0$. Due to the Minkowski determinant theorem, we conclude that $\det(a_{11}(\mu)a_{11}^*(\mu)) = 0$ and $\det(a_{12}(\mu)a_{12}^*(\mu))=0$. But this is contradicting the fact that $\mathcal{A}(\mu)$ is $\mathscr{J}_m$-unitary.
 \end{examp}
Note that the previous example implies that the elementary Blaschke-Potapov factors (see, \cite{ArD08}, chapter $4.2$) of first and second kind are de Branges matrices.\\
Now, in the following theorem, we describe a representation of the de Branges matrices using the factorization of matrix valued meromorphic functions discussed in the preceding section. This representation connects a de Branges matrix to a de Branges space and its associated function.
\begin{thm}
Let $\mathcal{A}(z)$ be a de Branges matrix of the form $(\ref{Basic matrix})$ and $U(z)=\mathcal{A}(z)M$ is of the form $(\ref{Secondary matrix})$. Then the following implications hold:
\begin{enumerate}
\item $\mathcal{A}(z)$ can have the following representation
\begin{equation}
\label{Real representation}
\mathcal{A}(z)=\begin{bmatrix}
    S(z)^{-1} & 0 \\
    0  & S(z)^{-1} \\
\end{bmatrix}
\begin{bmatrix}
    \tilde{a}_{11}(z) & \tilde{a}_{12}(z) \\
    \tilde{a}_{21}(z)  & \tilde{a}_{22}(z) \\
\end{bmatrix},
\end{equation}
where $\tilde{a}_{rt}(z)$ are $n\times n$ matrix valued entire functions and $S(z)$ is of the form $(\ref{Product})$.
\item The entire $n\times 2n$ matrix valued function $\mathfrak{E}(z)=[E_-(z)~E_+(z)]$, where
\begin{equation}
E_+(z)=\tilde{a}_{11}(z)+i\tilde{a}_{12}(z)~ \mbox{and} ~E_-(z)=\tilde{a}_{11}(z)-i\tilde{a}_{12}(z),
\end{equation}
generates a de Branges space $\mathcal{B}(\mathfrak{E})$.
\item $S(z)$ is an associated function of the de Branges space $\mathcal{B}(\mathfrak{E})$.
\end{enumerate}
\end{thm}
\begin{proof}
Due to Theorem \ref{Factorization}, the existence of the $n\times n$ matrix valued entire function $S(z)$ is evident. Now, $(1)$ follows after letting $\tilde{a}_{rt}(z)=S(z)~a_{rt}(z)$ for $r,t\in\{1,2\}$. Since
$$E_+(z)=i\sqrt{2}S(z)u_{12}(z)~\mbox{and}~E_-(z)=-i\sqrt{2}S(z)u_{11}(z),$$
$(2)$ follows from the assertions $(2)$ and $(5)$ of Lemma \ref{Meromorphic case}. To show that $S(z)$ is an associated function of the de Branges space $\mathcal{B}(\mathfrak{E})$, it is sufficient to show that 
$$\frac{E_+^{-1}S}{\rho_i}\in H_{n\times n}^2\hspace{.3cm}\mbox{and}\hspace{.3cm}\frac{S^{\#}\{E_-^{\#}\}^{-1}}{\rho_i}\in H_{n\times n}^2.$$
Now, $(3)$ follows from the assertions $(4)$ and $(6)$ of Lemma \ref{Meromorphic case}.
\end{proof}
We conclude this section with another example of de Branges matrices derived from a particular class of compact operators in a separable Hilbert space $\mathfrak{X}$ in terms of their characteristic matrix functions. The characteristic matrix functions are crucial in the theory of nonselfadjoint operators in Hilbert spaces. A detailed study of them can be found in \cite{Brodskii}. We consider $\mathscr{K}_0$ as the family of compact operators $T$ in $\mathfrak{X}$ having the following additional conditions:
\begin{enumerate}
\item The imaginary part $\frac{T-T^*}{2i}$ of $T$ is of rank $m=2n$.
\item $T$ is a non-dissipative operator.
\item $T$ has no real eigenvalues.
\end{enumerate}
Note that a similar family of operators was considered in \cite{Gubreev} to deal with a functional model problem in connection with the de Branges spaces of scalar valued entire functions. Observe that every operator $T\in \mathscr{K}_0$ is a completely nonselfadjoint operator in $\mathfrak{X}$. Otherwise, it will contradict the last condition in the definition of $\mathscr{K}_0$ .
\begin{examp}
Let $T\in \mathscr{K}_0$ be any operator, and the signature matrix $\mathscr{J}_m=(\mathscr{J}_{rt})$ as defined earlier. Then we have the following representation (see, \cite{Livshits}, Chapter I)
\begin{equation}
\frac{T-T^*}{i}u=\sum_{r,t=1}^m\langle u,u_r\rangle_\mathfrak{X}\mathscr{J}_{rt}~u_t,
\end{equation}
where $u, u_1, \ldots, u_m\in\mathfrak{X}$. Also, the characteristic matrix function of $T$ is given by
\begin{equation}
W_T(z)=I_m+iz(\langle (I_\mathfrak{X}-zT)^{-1}u_r,u_t\rangle_\mathfrak{X})\mathscr{J}_m.
\end{equation}
We claim that $W_T(z)$ is a de Branges matrix. Since $T$ is a compact operator, the characteristic matrix function $W_T(z)$ is meromorphic in $\mathbb{C}$. Also, it can be proved that $W_T(z)$ belongs to $U(\mathscr{J}_m)$. We consider 
$$W_T(z)=\begin{bmatrix}
    w_{11}(z) & w_{12}(z) \\
    w_{21}(z)  & w_{22}(z) \\
\end{bmatrix}
\hspace{.2cm}\mbox{and}\hspace{.2cm}\Phi(z)=[w_{22}(z)-iw_{21}(z)]~[w_{11}(z)+iw_{12}(z)]^{-1},
$$
where $w_{rt}$ are matrices of order $n\times n$ for $r,t\in\{1,2\}$. Due to condition $(3)$ in the definition of $\mathscr{K}_0$, we conclude that $W_T(z)$ is holomorphic in $\mathbb{R}$. Thus, as in example \ref{Example 1 }, $\Phi(z)$ is holomorphic in $\mathbb{R}$. This justifies our claim.
\end{examp}

\section{Integral representation of $\Phi(z)$ and parametrization of de Branges matrices}
\label{Sec-6}
In this section, we discuss an integral representation of the $n\times n$ matrix valued function  $\Phi(z)$ given by $(\ref{Main function})$ corresponding to a de Branges matrix of the form $(\ref{Basic matrix})$, which has a representation of the form $(\ref{Real representation})$. Additionally, we derive a parametrization of de Branges matrices based on the integral representation of $\Phi(z)$. Since $\Phi(z)\in \mathcal{C}^{n\times n}$, for all $z\in\mathbb{C}_+$, it has the following integral representation 
\begin{equation}
\Phi(z)=iQ-izP+\frac{1}{\pi i}\int_{-\infty}^\infty \left\{\frac{1}{x-z}-\frac{x}{1+x^2}\right\}~d\sigma(x),\label{Integral representation}
\end{equation}
where $Q=Q^*$, $P\succeq 0$ are $n\times n$ complex matrices and $\sigma(x)$ is a nondecreasing $n\times n$ matrix valued function on $\mathbb{R}$ such that $\int_{-\infty}^\infty \frac{d(\mbox{trace}~\sigma(x))}{1+x^2}<\infty$. We find the measure $d\sigma$ involved in the equation $(\ref{Integral representation})$. First we consider the matrix valued function 
\begin{equation}
\mathcal{A}^{\#}(z)=\mathcal{A}(\overline{z})^*=\mathscr{J}_m~\mathcal{A}(z)^{-1}\mathscr{J}_m=\begin{bmatrix}
    a_{11}^{\#}(z) & a_{21}^{\#}(z) \\
    a_{12}^{\#}(z)  & a_{22}^{\#}(z) \\
\end{bmatrix},
\end{equation}
which gives the inverse of $\mathcal{A}(z)$ as
\begin{equation}
\label{Inverse}
\mathcal{A}(z)^{-1}=\begin{bmatrix}
    a_{22}^{\#}(z) & -a_{12}^{\#}(z) \\
    -a_{21}^{\#}(z)  & a_{11}^{\#}(z) \\
\end{bmatrix}.
\end{equation}
Due to $(\ref{Inverse})$, $\Phi(z)$ can be rewritten as
\begin{equation}
\Phi(z)=[a_{11}^{\#}(z)+i a_{12}^{\#}(z)]^{-1}[a_{22}^{\#}(z)-i a_{21}^{\#}(z)],\label{Main function 1}
\end{equation}
and
\begin{equation}
\Phi^{\#}(z)=[a_{11}^{\#}(z)-i a_{12}^{\#}(z)]^{-1}[a_{22}^{\#}(z)+i a_{21}^{\#}(z)].
\end{equation}
Therefore, we have
\begin{align}
\frac{\Phi(z)+\Phi^{\#}(z)}{2}&=[a_{11}^{\#}(z)-i a_{12}^{\#}(z)]^{-1}[a_{11}(z)+ia_{12}(z)]^{-1}\nonumber\\
&=S^{\#}(z)\{E_-^{\#}(z)\}^{-1}E_+^{-1}(z)S(z),\label{Real part}
\end{align}
and
\begin{equation}
\frac{\Phi(z)-\Phi^{\#}(z)}{2i}=iS^{\#}(z)\{E_-^{\#}(z)\}^{-1}E_+^{-1}(z)S(z)-i\Phi(z).\label{Imaginary part}
\end{equation}
We also obtain
\begin{multline}
[a_{11}(z)+i a_{12}(z)]^*~[\Phi(z)+\Phi(z)^*]~[a_{11}(z)+i a_{12}(z)]\\
=\begin{bmatrix}
    I_n & -iI_n \\
\end{bmatrix}
[\mathscr{J}_m-\mathcal{A}(z)^*\mathscr{J}_m\mathcal{A}(z)]\begin{bmatrix}
    I_n  \\
    iI_n  \\
\end{bmatrix}
+2I_n.
\end{multline}
Thus for every $z\in\mathbb{C}$, where $\mathcal{A}(z)$ is holomorphic, we have
\begin{equation}
\label{Measure}
Re(\Phi(z)) \left\{
    \begin{array}{ll}
         \geq S(z)^*E_+^{-*}(z)E_+^{-1}(z)S(z),  & \mbox{if } z\in\mathbb{C}_+, \\
         =S(z)^*E_+^{-*}(z)E_+^{-1}(z)S(z), & \mbox{if } z\in\mathbb{R},\\
         \leq S(z)^*E_+^{-*}(z)E_+^{-1}(z)S(z), & \mbox{if} z\in\mathbb{C}_-.
    \end{array} \right.
\end{equation}
 Since $\Phi(z)$ is holomorphic on $\mathbb{R}$, due to $(\ref{Measure})$,
$$d\sigma(x)=S(x)^*E_+^{-*}(x)E_+^{-1}(x)S(x)dx,$$ 
and for all $z\in\mathbb{C}_+$
\begin{equation}
\Phi(z)=iQ-izP+\frac{1}{\pi i}\int_{-\infty}^\infty \left\{\frac{1}{x-z}-\frac{x}{1+x^2}\right\}~S(x)^*E_+^{-*}(x)E_+^{-1}(x)S(x)dx\label{Upper half plane}
\end{equation}
Also, due to $(\ref{Real part})$, we get the following representation of $\Phi(z)$ for all $z\in\mathbb{C}_-$
\begin{multline}
\Phi(z)=iQ-izP+\frac{1}{\pi i}\int_{-\infty}^\infty \left\{\frac{1}{x-z}-\frac{x}{1+x^2}\right\}~S(x)^*E_+^{-*}(x)E_+^{-1}(x)S(x)dx\\ + S^{\#}(z)\{E_-^{\#}(z)\}^{-1}E_+^{-1}(z)S(z).\label{Lower half plane}
\end{multline}
Given a de Branges matrix of the form $(\ref{Real representation})$, we can recover $\tilde{a}_{rt}$ for $r,t\in\{1,2\}$ in terms of $E_+(z)$, $E_-(z)$, $S(z)$ and $\Phi(z)$. It is immediate that
\begin{equation}
\tilde{a}_{11}(z)=\frac{E_+(z)+E_-(z)}{2}\hspace{.3cm}\mbox{and}\hspace{.3cm}\tilde{a}_{12}(z)=\frac{E_+(z)-E_-(z)}{2i}.\label{Upper entries}
\end{equation}
Due to $(\ref{Main function})$, we have
$$\tilde{a}_{22}(z)-i\tilde{a}_{21}(z)=S(z)\Phi(z)S(z)^{-1}[\tilde{a}_{11}(z)+i\tilde{a}_{12}(z)],$$
and due to $(\ref{Main function 1})$, we have
$$\tilde{a}_{22}(z)+i\tilde{a}_{21}(z)=S(z)\Phi^{\#}(z)S(z)^{-1} [\tilde{a}_{11}(z)-i\tilde{a}_{12}(z)],$$
which give together
\begin{multline*}
\tilde{a}_{22}(z)=S(z)\left[\frac{\Phi(z)+\Phi^{\#}(z)}{2}\right]S(z)^{-1}\tilde{a}_{11}(z)\\-S(z)\left[\frac{\Phi(z)-\Phi^{\#}(z)}{2i}\right]S(z)^{-1}\tilde{a}_{12}(z),
\end{multline*}
and
\begin{multline*}
\tilde{a}_{21}(z)=-S(z)\left[\frac{\Phi(z)+\Phi^{\#}(z)}{2}\right]S(z)^{-1}\tilde{a}_{12}(z)\\-S(z)\left[\frac{\Phi(z)-\Phi^{\#}(z)}{2i}\right]S(z)^{-1}\tilde{a}_{11}(z).
\end{multline*}
Now using $(\ref{Real part})$ and $(\ref{Imaginary part})$ in the previous two equations, we finally get
\begin{equation}
\tilde{a}_{22}(z)=S(z)S^{\#}(z)\{E_-^{\#}(z)\}^{-1}E_+^{-1}(z)E_-(z)+iS(z)\Phi(z)S(z)^{-1}\tilde{a}_{12}(z),\label{Lower right}
\end{equation}
and 
\begin{equation}
\tilde{a}_{21}(z)=-iS(z)S^{\#}(z)\{E_-^{\#}(z)\}^{-1}E_+^{-1}(z)E_-(z)+iS(z)\Phi(z)S(z)^{-1}\tilde{a}_{11}(z).\label{Lower left}
\end{equation}
The following theorem parametrizes de Branges matrices under consideration.
\begin{thm}
Given a de Branges matrix of the form $(\ref{Real representation})$, the following conclusions can be noted
\begin{enumerate}
\item The $n\times 2n$ matrix valued function $\mathfrak{E}(z)=[E_-(z)~E_+(z)]$, where $E_+(z)=\tilde{a}_{11}(z)+i\tilde{a}_{12}(z)$ and $E_-(z)=\tilde{a}_{11}(z)-i\tilde{a}_{12}(z)$ generates a de Branges space $\mathcal{B}(\mathfrak{E})$.
\item $S(z)$ is an associated function of the de Branges space $\mathcal{B}(\mathfrak{E})$.
\item The function $\Phi(z)$ of the form $(\ref{Main function})$ has the integral representation of the form $(\ref{Upper half plane})$ in the upper half plane and of the form $(\ref{Lower half plane})$ in the lower half plane with parameters $(P,Q)$ such that $P\succeq 0$ and $Q=Q^*$.
\end{enumerate}
Conversely, given such parameters $(P,Q)$ such that $P\succeq 0$ and $Q=Q^*$ along with a de Branges space $\mathcal{B}(\mathfrak{E})$ corresponding to an $n\times 2n$ matrix valued entire function $\mathfrak{E}(z)=[E_-(z)~E_+(z)]$ with an associated function $S(z)$, we can construct a de Branges matrix of the form $(\ref{Real representation})$.
\end{thm}
\begin{proof}
One side of the proof immediately follows from the previous discussion. Now, suppose $\mathcal{B}(\mathfrak{E})$ be a de Branges space corresponding to an $n\times 2n$ matrix valued entire function $\mathfrak{E}(z)=[E_-(z)~E_+(z)]$. Also, $S(z)$ is associated with $\mathcal{B}(\mathfrak{E})$. The given parameters $(P,Q)$ are such that $P\succeq 0$ and $Q=Q^*$. From the given information, we construct the function $\Phi(z)$ by using $(\ref{Upper half plane})$ and $(\ref{Lower half plane})$. Due to $(\ref{Associated function condition})$, we conclude that $\Phi(z)\in\mathcal{C}^{n\times n}$ and $\Phi(z)$ is holomorphic on $\mathbb{R}$ follows from Theorem \ref{Associated function} (see \cite{ArD08}, chapter $3.2$). $\Phi(z)$ is also satisfying $(\ref{Real part})$ and $(\ref{Imaginary part})$. We consider the following matrix valued function
$$\mathcal{A}(z)=\begin{bmatrix}
    S(z)^{-1} & 0 \\
    0  & S(z)^{-1} \\
\end{bmatrix}
\begin{bmatrix}
    \tilde{a}_{11}(z) & \tilde{a}_{12}(z) \\
    \tilde{a}_{21}(z)  & \tilde{a}_{22}(z) \\
\end{bmatrix},$$
where $\tilde{a}_{11}(z),~\tilde{a}_{12}(z)$ are defined by $(\ref{Upper entries})$ and $\tilde{a}_{21}(z),~\tilde{a}_{22}(z)$ are defined by $(\ref{Lower left})$ and $(\ref{Lower right})$ respectively. Now, it only remains to show that $\mathcal{A}(z)$ belongs to $U(\mathscr{J}_m)$. Again, we consider 
\begin{equation}
W(z)=M^*\mathcal{A}(z)M=\begin{bmatrix}
    w_{11}(z) & w_{12}(z) \\
    w_{21}(z)  & w_{22}(z) \\
\end{bmatrix}.
\end{equation}
Observe that $\mathcal{A}(z)$ belongs to $U(\mathscr{J}_m)$ if and only if $W(z)$ belongs to $U(j_m)$. Since $\mbox{Re}~\Phi(z)\succeq 0$ for all $z\in\mathbb{C}_+$,
$$w_{22}(z)=\frac{1}{2}[I_n+\Phi(z)]~[a_{11}(z)+ia_{12}(z)]$$
is invertible in $\mathbb{C}_+$ and almost everywhere on the real line. Thus the Potapov-Ginzburg transform of $W(z)$ is 
\begin{equation}
\mbox{PG}(W)(z)=\begin{bmatrix}
    w_{11}(z)-w_{12}(z)w_{22}(z)^{-1}w_{21}(z) & w_{12}(z)w_{22}(z)^{-1} \\
    -w_{22}(z)^{-1}w_{21}(z)  & w_{22}(z)^{-1} \\
\end{bmatrix}.
\end{equation}
Also, $W(z)$ is $j_m$-inner if and only if the Potapov-Ginzburg transform $\mbox{PG}(W)(z)$ belongs to $\mathcal{S}^{m\times m}_{in}$. First, we show that $W(z)$ is $j_m$-unitary almost everywhere on $\mathbb{R}$. For this purpose, we consider for $x\in\mathbb{R}$:
$$W(x)j_mW(x)^*-j_m=\begin{bmatrix}
    w_{11}w_{11}^*-w_{12}w_{12}^*-I_n & w_{11}w_{21}^*-w_{12}w_{22}^* \\
    w_{21}w_{11}^*-w_{22}w_{12}^*  & w_{21}w_{21}^*-w_{22}w_{22}^*+I_n \\
\end{bmatrix}.
$$
Now, for any $x\in\mathbb{R}$ where $\det S(x)\neq 0$, using $(\ref{Equality on real line})$, the following identities can be noted immediately
$$w_{11}(x)w_{11}(x)^*=\left[\frac{I_n+\Phi(x)^*}{2}\right]S(x)^{-1}E_+(x)E_+(x)^*S(x)^{-*}\left[\frac{I_n+\Phi(x)}{2}\right];$$
$$w_{22}(x)w_{22}(x)^*=\left[\frac{I_n+\Phi(x)}{2}\right]S(x)^{-1}E_+(x)E_+(x)^*S(x)^{-*}\left[\frac{I_n+\Phi(x)^*}{2}\right];$$
$$w_{12}(x)w_{12}(x)^*=\left[\frac{I_n-\Phi(x)}{2}\right]S(x)^{-1}E_+(x)E_+(x)^*S(x)^{-*}\left[\frac{I_n-\Phi(x)^*}{2}\right];$$
$$w_{21}(x)w_{21}(x)^*=\left[\frac{I_n-\Phi(x)^*}{2}\right]S(x)^{-1}E_+(x)E_+(x)^*S(x)^{-*}\left[\frac{I_n-\Phi(x)}{2}\right];$$
$$w_{11}(x)w_{21}(x)^*=-\left[\frac{I_n+\Phi(x)^*}{2}\right]S(x)^{-1}E_+(x)E_+(x)^*S(x)^{-*}\left[\frac{I_n-\Phi(x)}{2}\right];$$
and
$$w_{12}(x)w_{22}(x)^*=-\left[\frac{I_n-\Phi(x)}{2}\right]S(x)^{-1}E_+(x)E_+(x)^*S(x)^{-*}\left[\frac{I_n+\Phi(x)^*}{2}\right].$$
Now, using $(\ref{Measure})$, it can be proved that $W(z)$ is $j_m$-unitary almost everywhere on $\mathbb{R}$, i.e., for almost every $x\in\mathbb{R}$ the following identities hold
$$w_{11}(x)w_{11}(x)^*-w_{12}(x)w_{12}(x)^*=I_n;~w_{21}(x)w_{21}(x)^*-w_{22}(x)w_{22}(x)^*=-I_n$$
and
$$w_{11}(x)w_{21}(x)^*-w_{12}(x)w_{22}(x)^*=0.$$
Here, we only show the calculation for the first identity, and the remaining can be done similarly. Now, for any $x\in\mathbb{R}$ where $\det S(x)\neq 0$ and $\det E_+(x)\neq 0$, we have
\begin{align*}
&w_{11}(x)w_{11}(x)^*-w_{12}(x)w_{12}(x)^*\\
=&\left[\frac{I_n+\Phi(x)^*}{2}\right]S(x)^{-1}E_+(x)E_+(x)^*S(x)^{-*}\left[\frac{I_n+\Phi(x)}{2}\right]\\
&-\left[\frac{I_n-\Phi(x)}{2}\right]S(x)^{-1}E_+(x)E_+(x)^*S(x)^{-*}\left[\frac{I_n-\Phi(x)^*}{2}\right]\\
=&\left[\frac{I_n+\Phi(x)^*}{2}\right]S(x)^{-1}E_+(x)E_+(x)^*S(x)^{-*}\left[\frac{I_n+\Phi(x)}{2}-\frac{I_n-\Phi(x)^*}{2}\right]\\
&+\left[\frac{I_n+\Phi(x)^*}{2}-\frac{I_n-\Phi(x)}{2}\right]S(x)^{-1}E_+(x)E_+(x)^*S(x)^{-*}\left[\frac{I_n-\Phi(x)^*}{2}\right]\\
=&\left[\frac{I_n+\Phi(x)^*}{2}\right]S(x)^{-1}E_+(x)E_+(x)^*S(x)^{-*}\left[\frac{\Phi(x)+\Phi(x)^*}{2}\right]\\
&+\left[\frac{\Phi(x)+\Phi(x)^*}{2}\right]S(x)^{-1}E_+(x)E_+(x)^*S(x)^{-*}\left[\frac{I_n-\Phi(x)^*}{2}\right]\\
=&\left[\frac{I_n+\Phi(x)^*}{2}\right]S(x)^{-1}E_+(x)E_+(x)^*S(x)^{-*}S(x)^*E_+(x)^{-*}E_+(x)^{-1}S(x)\\
&+S(x)^*E_+(x)^{-*}E_+(x)^{-1}S(x)S(x)^{-1}E_+(x)E_+(x)^*S(x)^{-*}\left[\frac{I_n-\Phi(x)^*}{2}\right]\\
=&\left[\frac{I_n+\Phi(x)^*}{2}\right]+\left[\frac{I_n-\Phi(x)^*}{2}\right]=I_n.
\end{align*}
Since $W(z)$ is $j_m$-unitary almost everywhere on the real line, the Potapov-Ginzburg transform $\mbox{PG}(W)(z)$ is unitary on the real line. Now, we will apply Smirnov maximum principle for matrix valued functions to show that $\mbox{PG}(W)(z)\in \mathcal{S}^{m\times m}_{in}$. Here, we only need to show that $\mbox{PG}(W)(z)$ belongs to the Smirnov class. It is sufficient to show that the four blocks of $\mbox{PG}(W)(z)$ belong to the Smirnov class $\mathcal{N}_+^{n\times n}$. We consider $$c(z)=[\Phi(z)-I_n]~[\Phi(z)+I_n]^{-1}.$$
Since $\mbox{Re}~\Phi(z)\succeq 0$ for all $z\in\mathbb{C}_+$, we have $||c(z)||\leq 1$ for all $z\in\mathbb{C}_+$. Now, the $(1,2)$ block of $\mbox{PG}(W)(z)$ is of the form
$$w_{12}(z)w_{22}(z)^{-1}=c(z)$$
belongs to $\mathcal{N}_+^{n\times n}$ as $||c(z)||\leq 1$ for all $z\in\mathbb{C}_+$. The $(2,2)$ block of $\mbox{PG}(W)(z)$ is of the form
\begin{align*}
w_{22}(z)^{-1}&=2E_+(z)^{-1}S(z)[I_n+\Phi(z)]^{-1}\\
&=E_+(z)^{-1}S(z)[I_n-c(z)]
\end{align*}
belongs to $\mathcal{N}_+^{n\times n}$ as $E_+(z)^{-1}S(z)\in \mathcal{N}_+^{n\times n}$ and $I_n-c(z)$ is bounded. The $(1,1)$ block of $\mbox{PG}(W)(z)$ is of the form
\begin{align*}
&w_{11}(z)-w_{12}(z)w_{22}(z)^{-1}w_{21}(z)\\
=&\frac{1}{2}\left\{[I_n+\Phi^{\#}(z)-[I_n-\Phi(z)][I_n+\Phi(z)]^{-1}[I_n-\Phi^{\#}(z)]\right\}S(z)^{-1}E_-(z)\\
=&\left\{\left[\frac{\Phi(z)+\Phi^{\#}(z)}{2}\right]+[I_n-\Phi(z)][I_n+\Phi(z)]^{-1}\left[\frac{\Phi(z)+\Phi^{\#}(z)}{2}\right]\right\}S(z)^{-1}E_-(z)\\
=&\left\{I_n+[I_n-\Phi(z)][I_n+\Phi(z)]^{-1}\right\}S^{\#}(z)\{E_-^{\#}(z)\}^{-1}E_+^{-1}(z)E_-(z)\\
=&[I_n-c(z)]S^{\#}(z)\{E_-^{\#}(z)\}^{-1}E_+^{-1}(z)E_-(z)
\end{align*}
belongs to $\mathcal{N}_+^{n\times n}$ as $S^{\#}(z)\{E_-^{\#}(z)\}^{-1}\in\mathcal{N}_+^{n\times n}$ and $E_+^{-1}(z)E_-(z)\in \mathcal{S}_{in}^{n\times n}$. The $(2,1)$ block of $\mbox{PG}(W)(z)$ is of the form
\begin{align*}
&-w_{22}(z)^{-1}w_{21}(z)\\
=&E_+^{-1}(z)S(z)[I_n+\Phi(z)]^{-1}[I_n-\Phi^{\#}(z)]S(z)^{-1}E_-(z)\\
=&E_+^{-1}(z)S(z)[I_n+\Phi(z)]^{-1}S(z)^{-1}E_-(z)\\
&-E_+^{-1}(z)S(z)[I_n+\Phi(z)]^{-1}\Phi^{\#}(z)S(z)^{-1}E_-(z)\\
=&E_+^{-1}(z)E_-(z)-E_+^{-1}(z)S(z)[I_n+\Phi(z)]^{-1}[\Phi(z)+\Phi^{\#}(z)]S(z)^{-1}E_-(z)\\
=&E_+^{-1}(z)E_-(z)-E_+^{-1}(z)S(z)[I_n-c(z)]S^{\#}(z)\{E_-^{\#}(z)\}^{-1}E_+^{-1}(z)E_-(z).
\end{align*}
belongs to $\mathcal{N}_+^{n\times n}$ as $E_+^{-1}(z)E_-(z)\in \mathcal{S}_{in}^{n\times n}$, $I_n-c(z)$ is bounded and $E_+^{-1}(z)S(z)$, $S^{\#}(z)\{E_-^{\#}(z)\}^{-1}$ belong to $\mathcal{N}_+^{n\times n}$. This completes the proof.
\end{proof}
Note that the de Branges matrix that we constructed in the previous theorem from given $S(z)$, $E_+(z)$, and $E_-(z)$ is unique subject to the given parameters $(P,Q)$. Now, suppose two de Branges matrices are 
$$ \mathcal{A}(z)=\begin{bmatrix}
    S(z)^{-1} & 0 \\
    0  & S(z)^{-1} \\
\end{bmatrix}
\begin{bmatrix}
    \tilde{a}_{11}(z) & \tilde{a}_{12}(z) \\
    \tilde{a}_{21}(z)  & \tilde{a}_{22}(z) \\
\end{bmatrix}$$
and 
$$\mathcal{B}(z)=\begin{bmatrix}
    S(z)^{-1} & 0 \\
    0  & S(z)^{-1} \\
\end{bmatrix}
\begin{bmatrix}
    \tilde{a}_{11}(z) & \tilde{a}_{12}(z) \\
    \tilde{b}_{21}(z)  & \tilde{b}_{22}(z) \\
\end{bmatrix}$$
corresponding to the parameters $(P,Q)$ and $(\tilde{P},\tilde{Q})$ respectively. Due to $(\ref{Lower left})$, we have
$$\tilde{b}_{21}(z)-\tilde{a}_{21}(z)=S(z)[(P-\tilde{P})+z(\tilde{Q}-Q)]S(z)^{-1}\tilde{a}_{11}(z).$$
Similarly, due to $(\ref{Lower right})$, we have
$$\tilde{b}_{22}(z)-\tilde{a}_{22}(z)=S(z)[(P-\tilde{P})+z(\tilde{Q}-Q)]S(z)^{-1}\tilde{a}_{12}(z).$$
Thus, the following identity holds
$$\mathcal{B}(z)=\begin{bmatrix}
    I_n & 0 \\
    (P-\tilde{P})+z(\tilde{Q}-Q)  & I_n \\
\end{bmatrix}
\mathcal{A}(z).$$


\vspace{.2in}
\noindent \textbf{Acknowledgements:} 
The research of the first author is supported by the University Grants Commission (UGC) fellowship (Ref. No. DEC18-424729), Govt. of India.
The research of the second author is supported by the MATRICS grant of SERB (MTR/2023/001324).

\vspace{.4cm}

\noindent\textbf{Conflict of interest:}\\
The authors declare that they have no conflict of interest.

\vspace{.4cm}

\noindent\textbf{Data availability:}\\
No data was used for the research described in the article.

\end{document}